\documentclass[a4paper, 12pt]{article}
\usepackage{graphics}
\usepackage{ifthen}
\usepackage{amsmath,amsfonts,amstext,amscd,bezier,amssymb,a4,enumerate,epsfig,psfrag,subfigure}
\usepackage[latin1]{inputenc}
\usepackage{graphicx}

\usepackage{amsthm,amsopn}

\usepackage{color} 

\everymath={\displaystyle}

\newcommand{\RR}{{\mathbb{R}}}

\newtheorem{lemma} {Lemma}
\newtheorem{prop} {Proposition}

\newtheorem{theorem} {Theorem}

\newtheorem*{theo*} {Theorem}

\renewcommand{\qed}{\hfill \mbox{\raggedright \rule{.1in}{.1in}}}

\newcommand{\dist}{\operatorname{dist}}

\newcommand{\re}[1]{(\ref{#1})}
\newcommand{\rn}{{\mathbb{R}^n}}
\newcommand{\hyph}{$\mathbf{(AP)_b}$ }

\title{Results of Ambrosetti-Prodi type \\
for non-selfadjoint elliptic operators}
\author{Boyan Sirakov, Carlos Tomei, and André Zaccur}

\date{}
\begin{document}

\maketitle


\begin{abstract}
The well-known Ambrosetti-Prodi theorem considers perturbations of the Dirichlet Laplacian by a nonlinear function whose derivative jumps over the principal eigenvalue of the operator. Various extensions of this landmark result were obtained for self-adjoint operators, in particular by Berger and Podolak, who gave a geometrical description of the solution set. In this text we show that similar theorems are valid for non self-adjoint operators.  In particular, we prove that the semilinear operator is a global fold. As a consequence, we obtain what appears to be the first exact multiplicity result for elliptic equations in non-divergence form. We employ techniques based on the maximum principle.
\end{abstract}






\section{Introduction}\label{chap: introduction}

In this paper we study the solvability of the equation
\begin{equation}\label{princ}
-Lu= f(u) + g(x)
\end{equation}
with a Dirichlet boundary condition in a  bounded $C^{1,1}$-domain in $\rn$, where $L$ is an uniformly elliptic operator in {\it non-divergence form} with bounded coefficients, $f$ is a nonlinear function whose behaviour at plus or minus infinity is different with respect to the first eigenvalue of $L$, and $g$ is a given fixed function.
Under such conditions the equation \re{princ} is usually named of Ambrosetti-Prodi type, in honor of the celebrated work \cite{AP}.


The large number of developments on Ambrosetti-Prodi type problems have gone, grosso modo, in two directions (more detailed statements and references will be given below): first, a precise count of solutions, a description of the solution set and of the action of the operator $-L-f(\cdot)$  on a natural function space are available if  $L$ is in divergence form, since then variational methods and theory of self-adjoint operators can be used; second, for more general operators $L$  only fixed-point methods are available, and they lead to partial existence results in which just a lower bound on the number of solutions is given, as well as an incomplete description of the solution set.

In the present work we bridge this apparent gap, and show that for any operator $L$ in non-divergence form with continuous second-order coefficients and for any nonlinearity $f$ whose derivative has range containing $\lambda_1(L)$ and contained in a determined interval around $\lambda_1(L)$, we can precisely count the solutions and describe the action of $-L-f(\cdot)$ on $W^{2,p}(\Omega)$, for any $p\ge n$. Our approach, which is (necessarily) different from those in the previous works, uses techniques based on the maximum principle as well as elliptic regularity and results on the first eigenvalue of non-divergence form operators obtained by Berestycki, Nirenberg and Varadhan in \cite{BNV}.

To our knowledge, Theorem \ref{theo: strictconvexInt} below  is the first result on exact multiplicity of solutions (i.e. exact number of solutions different from $0$ or $1$) for equations driven by an operator in non-divergence form.\medskip

Let us now give the detailed statement of our main result. We set
$$
F(u) = - L u - f(u),\qquad Lu:= a_{ij}\partial_i\partial_ju + b_i \partial_iu + cu=\mathrm{tr}(AD^2u)+b.\nabla u + cu,
 $$
 where the coefficients $A(x), b(x),c(x)$ satisfy the following assumptions: for some constants $\Lambda\ge\lambda>0$,
\[ A \in C(\overline{\Omega})  , \qquad \mathrm{spec}(A)\in [\lambda,\Lambda], \qquad |b|,|c|\le \Lambda.\]
We denote the principal  eigenvalue of $-L$ by $\lambda_1=\lambda_1(L,\Omega)\in \RR$ (necessarily simple, isolated)  and a positive associated eigenfunction by $\phi_1$ (see Section \ref{sec: spectraltheory}).

We consider {Lipschitz functions} $f:\RR \to \RR$ which satisfy the Ambrosetti-Prodi type hypothesis:

\noindent$\mathbf{(AP)_b}$ for some constants $a,b\in \mathbb{R}$, $a<\lambda_1<b$, $ a\leq \frac{f(x) - f(y)}{x-y} \leq b$ for $ x\neq y \ $, \\
and for some $M\ge0$, we have
$
f(s) \ge \max \{ bs  -M, as-M\}  \quad \hbox{ for all } \ s \in \RR.
$

\bigskip
Since the problem does not change if we replace $L$ by $L-a$, $f$ by $f-a$, and $b$  by $b-a$, we will assume without loss that $a=0$.

We also  assume some convexity of $f$.\smallskip

\noindent$\mathbf{(C)}$ The function $f$ is convex on $\mathbb{R}$. Also, $f$ is not in the form $f(s)=\lambda_1 s +\beta$, $\beta\in\RR$, in a left or a right neighbourhood of $s=0$.\smallskip

We set $X=\{u\in W^{2,p}(\Omega), \: u=0 \mbox{ on }\partial\Omega\}$, $Y=L^p(\Omega)$ and consider the maps $L,F:X \to Y$.
From now on, if $p>n$ when we say a constant depends on $L$ we will mean it depends on $n,p,\lambda, \Lambda$, and a modulus of continuity of the coefficient matrix $A$. When $p=n$ we have less control on the constants, and they may depend on $L$ in a more complicated way.

\begin{theorem}\label{theo: strictconvexInt} There exists $B = B(L,\Omega)>\lambda_1$ such that if $f$ satisfies $\mathbf{(C)}$ and \hyph with $b<B$,  then the operator $ F(u)=-Lu-f(u)$ is a global topological fold from $X$ to $Y$. More specifically, there exist (bi-Lipschitz) homeomorphisms $\Phi_1:X\to X$, $\Phi_2: Y\to Y$ and hyperplanes  $W\subset X$, $Z\subset Y$, such that  $X=W\oplus\RR\phi_1$, $Y=Z\oplus\RR\phi_1$, for which the restriction $(\Phi_2\circ F \circ \Phi_1)|_W $ is a homeomorphism, and
\begin{equation}\label{maine}
(\Phi_2\circ F \circ \Phi_1)(w+t\phi_1) = -Lw - |t|\phi_1,
\end{equation}
 for any $t\in \RR$, $w\in W$. For each  $w_0\in W$, $z_0\in Z$, the map $\Phi_2$ keeps the  line $\{z_0+t\phi_1, t\in \RR\}$ invariant, while the map $\Phi_1$ transforms $\{w_0+t\phi_1, t\in \RR\}$ into a curve which is asymptotically parallel to $\phi_1$ for large values of $|t|$ (in the sense of \re{fini3}, below).

\end{theorem}

In particular, the equation
$ F(u) = z_0 + t \phi_1 $
has exactly 2, 1 or 0 solutions in~$X$, according to whether $t$ is respectively smaller than, equal to, or larger than a real number $\bar t(z_0)$.

The hypotheses on $f$ are essentially optimal for this type of multiplicity, even in the simplest case  $L=\Delta$, $f\in C^2(\RR)$. Indeed, it is well-known that if $\overline{\mathrm{Im}(f^\prime)}$ does not meet the spectrum of $L$ then $F$ is a homeomorphism, whereas when Im($f^\prime$) contains more than one eigenvalue of $L$ then there may be more than two solutions for some right-hand sides (see for instance \cite{Am}, \cite{So}). Furthermore, in the recent work \cite{CTZ2} it is shown that, under \hyph, even if $f^{\prime\prime}$ is negative just at one point then there are right-hand sides $z_0 + t \phi_1$ admitting at least four solutions.\smallskip

We now discuss the main results on Ambrosetti-Prodi type problems obtained prior to Theorem
\ref{theo: strictconvexInt}.
The original theorem assumes that $f$ is a strictly convex $C^2$ function such that $f'(\RR) = (a,b)$ contains  the first but not the second eigenvalue of the Laplacian, see \cite{AP} and \cite{MM}. They prove that  the critical set $\mathcal{C}$ of $F$ with $L= \Delta $, defined from $\bar X=C^{2,\alpha}(\Omega)\cap C_0(\overline{\Omega})$ into $\bar Y=C^\alpha(\Omega)$,  is a hypersurface homeomorphic to a hyperplane, which splits $\bar X$ into two disjoint components $A,B$, i.e. $\bar X = A \cup \mathcal{C}\cup B$. Ambrosetti and Prodi show that $F$ is injective on $\mathcal{C}$ and $F(\mathcal{C})$ also generates a split of $\bar Y$, $\bar Y = S_0 \cup F(\mathcal{C}) \cup S_2$, into three connected components,   in such a way that $A$ and $B$ are taken by $F$  homeomorphically to $S_2$. Later Dancer \cite{Da},  Berestycki \cite{B}, de Figueiredo and Solimini \cite{dF}, \cite{dFS}, obtained extensions of that result  for self-adjoint second order operators in divergence form, giving characterizations of the sets $A$ and $B$ in terms of the Morse index of their elements as critical points of the energy functional, or of the coercivity of the associated linearized operator.

In  those works the focus was the decomposition of domain and counterdomain of $F$ in components on which the restriction of $F$ acts injectively. On the other hand, with a view on the solvability of the equation for a given right-hand side,
 Berger and Podolak \cite{BP}, and Berger and Church \cite{BC} used a global Lyapunov-Schmidt decomposition  to give a geometric description of the map $F$: it is a topological fold from $H^2(\Omega) \cap H_0^1(\Omega)$ into $ L^2(\Omega)$.
We note that the notion of a fold we use is a (global) Banach space version of the original concept introduced by Whitney  \cite{W}, \cite{MST2}, in his study of  generic maps from the plane to the plane.

In a nutshell, the works \cite{BP}, \cite{BC} rely on the fundamental fact that the ``vertical" lines $\{ z + t \phi_1, t\in \RR \} \subset Y$, $z\perp \phi_1$, when inverted by $F$, give rise to very special curves in the domain of $F$, the fibers. This follows from the global Lyapunov-Schmidt decomposition for $F$. Extensions of this  approach (in \cite{TZ}, \cite{CTZ1} and the references therein), still in the self-adjoint case, allow for less differentiability on the nonlinearity $f$, together with a larger choice of operators $L$. Thus, for example, folds are obtained for Schr\"odinger operators in bounded and unbounded domains, including the hydrogen atom and the quantum harmonic oscillator, the spectral (self-adjoint) fractional Laplacian. \smallskip

When $L$ is not self-adjoint, the convenient spectral estimates and integral representation of the equation are not available for constructing the Lyapunov-Schmidt decomposition.   Prior to this work only topological  methods, more precisely fixed-point theorems for Banach spaces, have been applied to nondivergence form equations (see \cite{AH}, \cite{dFSi}, \cite{M2}, \cite{De}, \cite{Si}, and the references in these papers; for a different approach to ODEs, see \cite{MST2}). Topological methods cover a very large scope of problems, such as fully nonlinear equations or systems of equations, but have the important drawback that no exact count of solutions can be obtained and we are left with a rather poor description of the solvability of $F(u)=g$ for different right-hand sides. Specifically, these results always state that for every given right-hand side $z_0 + t \phi_1$ the problem has at least 2, at least~1 or 0 solutions according to whether $t$ is respectively less, equal or larger than a real number~$\bar t(z_0)$.
\smallskip

In Section \ref{sec: LS} we will construct a  global Lyapunov-Schmidt decomposition for~$F$, for the first time in a non-divergence setting. The core of the construction is an elliptic estimate which can be interpreted as a bi-Lipschitz bound of~$F$ on ``horizontal" subspaces of $X$,  which is uniform in the ``heights" of these subspaces (Proposition~\ref{lemma: coercive} and its consequence Theorem \ref{theo: Phiint}, below). This estimate allows  us to construct fibers and to prove basic properties about their geometry and asymptotic behavior at infinity, implying also the properness of $F$.

The Lyapunov-Schmidt decomposition may be taken as a robust starting point for numerics, following ideas developed for the self-adjoint case in \cite{CT}, \cite{K},  but we do not handle the issue in this paper. We could also allow less regular domains, for instance domains $\Omega$ that satisfy an exterior cone condition, but will also not consider these technicalities here.

The last Section \ref{chap: convexLipschitz} is dedicated to the proof of Theorem \ref{theo: strictconvexInt}. We deviate from and simplify the earlier approaches, which identify critical points of $F$ by computing derivatives of the so-called height function along fibers.
Here we observe that our assumptions and properties of the principal eigenvalue of $L$ and its positive eigenfunction suffice to prove that no point in the image of $F$ has three preimages. From  the existence of fibers and the properness of $F$, the fold structure is then deduced from the more general Proposition \ref{fold}.

Finally, we remark that when $L$ is self-adjoint, the optimal value for $B$ in Theorem \ref{theo: strictconvexInt} is the second eigenvalue of $L$. On the other hand, a more general $L$ might not even have a second real eigenvalue, and such a simple explicit lower bound for $B$ is not available. However, our proof does imply an explicit lower bound on $B$, depending on $L$ and $\Omega$, in terms of the constants which appear in the basic estimates of the elliptic theory.

\section{The global Lyapunov-Schmidt decomposition}\label{sec: LS}

In this text, the letters $C$ and $c$, indexed or not, denote positive constants which depend on the appropriate quantities and may change from line to line.\smallskip

The convexity of the nonlinearity $f$ plays no role in most of this section.

\subsection{Preliminaries, basic results on principal eigenvalues}\label{sec: spectraltheory}

Let
$L$ be as in the introduction, and $\Omega$ be an arbitrary domain.
We recall some basic facts about $L$ from \cite{BNV}, related to maximum principles. The {\it principal eigenvalue} $\lambda_1(L,\Omega)$ is defined by
\[
\lambda_1=\lambda_1(L,\Omega)=\sup \left\{\lambda \in \RR : \exists \phi \in W^{2,n}_{\mathrm{loc}}(\Omega) \mbox{ such that}  \begin{array}{rc}
 (L+ \lambda)\phi \leq 0  &  \mbox{ in } \Omega \\
   \phi > 0 &  \mbox{ in } \Omega
      \end{array}
         \right\} .
\]

The associated eigenspace is spanned by the eigenvector $\phi_1=\phi_1(L,\Omega)>0$. It is also known that for the dual operator $L^*:Y^* =L^{\frac{n}{n-1}}(\Omega)\to X^* $, we have $\lambda_1(L^*,\Omega)=\lambda_1(L,\Omega)$ and $\phi_1^*=\phi_1(L^*,\Omega)>0$ (see \cite{IB}).

\begin{theorem}[Theorem 2.3, \cite{BNV}]\label{theo: changesign} Let $\phi$ be an eigenfunction of $-L$ with eigenvalue $\lambda \ne \lambda_1$.
Then (i) $Re (\lambda) > \lambda_1$, (ii) If $\phi $ is real, then it changes sign in $\Omega$.
\end{theorem}

The following existence and uniqueness result holds.

\begin{theorem}[Theorem 1.2, \cite{BNV}, and Theorem 9.13 in \cite{GT}]\label{theo: exist} If $\lambda_1>0$ then the map $L:X\to Y$ is an homeomorphism, and if $Lu=h$ then
$$
\|u\|_X \le \frac{C_{ABP}}{\lambda_1} \|h\|_Y
$$
where $C_{ABP}$ depends on $n$, $\lambda$, $\Lambda$, and $\Omega$. If $h\le0$, $h\not\equiv0$ in $\Omega$ then $u>0$ in $\Omega$.
\end{theorem}

We will use the following characterization of $\lambda_1$.

\begin{prop}[Corollary 1.1, \cite{BNV}]\label{eqdef} If for some $A\in \mathbb{R}$ there exists a bounded $\phi\in W^{2,n}_{\mathrm{loc}}(\Omega)$ such that
$ (L+ A)\phi \geq 0 $ in $\Omega$,
   $\limsup_{x\to\partial\Omega}\phi(x) \leq 0 $, and $\phi$ is positive somewhere in $\Omega$, then $$\lambda_1(L,\Omega)\le A.$$
\end{prop}

The principal eigenvalue increases together with the zero-order coefficient of the operator, and decreases when the domain enlarges.

\begin{prop}[Proposition 2.1, \cite{BNV}]\label{prop: monotonicity} If $ V(x) \ge 0, V \not\equiv 0$ is a bounded function in $\Omega$, then $\lambda_1(L+ V, \Omega) > \lambda_1 (L, \Omega)$.
\end{prop}

\begin{theorem}[Theorem 2.4, \cite{BNV}]\label{theo: lambda1measure}
Let $\Omega'\subset \Omega$ be an open subset  and $\delta >0$ be such that $|\Omega'|\leq |\Omega|-\delta$. Then there exists $\eta=\eta(L,\Omega,\delta)>0$ such that
\[ \lambda_1(L,\Omega') - \lambda_1(L,\Omega) \geq \eta.\]
\end{theorem}

In the sequel we will need the following fact.

\begin{prop}\label{prop: principal}
There exists a constant $\widetilde{B} = \widetilde{B}(L,\Omega)>\lambda_1(L,\Omega)$ such that for any bounded function $V(x)$ with  $V(x) \leq \widetilde{B}$, the operator $\tilde L u = L u + V u$ has a nontrivial kernel if and only if $0$ is the principal eigenvalue of $\tilde L$.
\end{prop}
\begin{proof} We take $\widetilde{B} = \lambda_1(L,\Omega) + \eta$, where  $\eta$ is obtained from Theorem \ref{theo: lambda1measure}  with $\delta = |\Omega|/2$. Then for every open set $\Omega' \subset \Omega$ satisfying $|\Omega'|\leq  |\Omega|/2$, $\lambda_1(L, \Omega') > \widetilde{B}$.

Suppose $V(x)\leq \widetilde{B}$ and
 $u\not\equiv 0$ with $Lu + V(x)u = 0$ in $\Omega$, and $u=0$ on $\partial\Omega$. We show that $u$ does not change sign in $\Omega$ (and so by Theorem \ref{theo: changesign} it is a principal eigenfunction). Define
\[\Omega^+_u := \{ x\in \Omega : u(x) > 0\} \ , \quad \Omega^-_u :=\{ x\in \Omega : u(x) < 0\},\]
and assume by contradiction that  $\Omega_u^+ , \Omega^-_u \ne \Omega$. At least one of the sets  $\Omega_u^+ , \Omega^-_u $ (say $\Omega_u^+$) has measure smaller or equal to  $ |\Omega|/2$. Use Proposition \ref{eqdef} with $\Omega$ replaced by $\Omega_u^+$, and $\phi=u$ to obtain
\[
Lu + \widetilde{B} u  \geq Lu +V(x)u =0, \qquad u>0\quad \mbox {in }\Omega_u^+,\qquad
\mbox{and}\qquad \limsup_{x\to\partial \Omega_u^+} u(x)=0\] and hence $\lambda_1(L,\Omega^+_u)\le \widetilde{B}$, a contradiction.
\qed
\end{proof}

\bigskip
We quote a quantitative Hopf lemma  \cite{Si1}, which extends results by Brezis-Cabre \cite{BCa} for $L=\Delta$, and by Krylov \cite{Kr}, who obtained the interior estimate.

\begin{theorem}[Theorem 3.1, \cite{Si1}]\label{qsmp} There exist $\varepsilon,c>0$ depending on $n$,  $\lambda$, $\Lambda$, $p$, and $\Omega$ such that, for each  solution $u\in W^{2,p}(\Omega)$, $p>n$, of
$Lu\le 0$, $u\ge0$ in $\Omega$,
$$
\inf_{\Omega} \frac{u}{d} \ge c\left( \int_{\Omega} (-Lu)^{\varepsilon}\right)^{1/{\varepsilon}},
$$
where $d(x)=\mathrm{dist}(x,\partial \Omega)$.
\end{theorem}

\subsection{The decomposition}\label{sec: DLS}

We decompose $X$ and $Y$ in direct sums of {\it horizontal} and {\it slanted} subspaces,
\[X = W \oplus V \ , \quad Y = Z \oplus V , \]
where $Z = \mathrm{vect}( \phi_1^*)^{\perp}=(\mathbb{R}\phi_1^*)^\perp$, $W = Z \cap X$, $V = \mathrm{vect}( \phi_1^*) \subset X, Y$. Clearly $Z \cap V = \{0\}$, as $\phi_1>0$. Contrary to the case when $L$ is a self-adjoint operator, this decomposition is not necessarily orthogonal with respect to the  inner product in $L^2(\Omega)$.  For each $g\in Y$ we split
\[ g=Pg + (I-P)g=z_g + h_g\phi_1 \in Z \oplus V \  \]
where $P$ is the projection $P:Z \oplus V =Y \to Z \oplus V ,  \ z + v \mapsto z$.
Throughout the text, the letters $w$ and $z$ will be reserved, respectively, for elements of the horizontal spaces $W$, $Z$.
From the closed graph theorem, the norms on $X$ and $Y$  are equivalent to the direct sum norms,  $\|g\|_Y \cong \|Pg\|_{Z} + \|(I-P)g\|_V$, and we change from one norm to the other without warning.

For each $g\in Y$, $g=z_g + h_g \phi_1$, we write the equation $F(u) = g$ as
$$F(w + t \phi_1) = z_g + h_g \phi_1\qquad \mbox{for the unknowns} \quad w\in W \, , \ t\in \RR.$$
 For each fixed $t\in \RR$ we set $F_t(w)= F(w+t\phi_1)$ and decompose the equation we want to solve as follows,
\begin{equation}\label{eq: system}
 \left\{ \begin{array}{l}
 PF_t(w) = z_g \\
 (I-P)F_t(w)=h_g\phi_1
       \end{array} \right. \ \mbox{for the unknowns} \quad w\in W \, , \ t\in \RR.
\end{equation}
We will show in Proposition \ref{lemma: diffeo} below that
the maps $PF_t:W \to Z$  are  bi-Lipschitz homeomorphisms, uniformly in $t\in \RR$ (bi-Lipschitz means the inverse is also Lipschitz). We may thus solve the first equation in \re{eq: system}; then from the second equation in \re{eq: system} we can write $h_g$ in terms of $z_g$ and $t$.

\begin{theorem}\label{theo: Phiint}  There exists $B = B(L,\Omega)>\lambda_1$ such that if $f$ satisfies \hyph with $b<B$,  then
the operator $ F=-L-f(\cdot)$ admits a global Lyapunov-Schmidt decomposition:
the map $\Psi :X\to Y$ defined by
\[
\Psi(w+t\phi_1) = PF_t(w)+t\phi_1 ,
\]
is a bi-Lipschitz homeomorphism and, identifying $Z\oplus V$ and $Z\times\RR$,
\[ \tilde F = F \circ \Psi^{-1} : Z \times \RR \to Z \times \RR \qquad\mbox{is} \qquad  \tilde F(z,t) = (z, \tilde h(z,t)),\]
where $\tilde h $ is the  Lipschitz function given by
\[\tilde h(z,t):=\frac{\langle F(\Psi^{-1}(z+t\phi_1)),\phi_1^*\rangle}{\langle \phi_1,\phi_1^*\rangle}.\]
\end{theorem}

\bigskip

%
%

The following crucial coercivity bound for $\Psi$ is proved in Section \ref{sec: theoremcoercive} below.

\begin{prop}\label{lemma: coercive} Let $\Psi$ be as in the previous theorem.
 There exists a constant $B = B(L,\Omega)>\lambda_1$ such that if $f$ satisfies  \hyph with $b<B$,  then for some $c=c(L,\Omega)>0$ we have
\[\| \Psi(u) - \Psi(\tilde u)\|_Y \geq c \| u - \tilde u \|_X \ \quad \hbox{for all} \ u, \tilde u \in X .\]
\end{prop}

 We now turn to the proof of Theorem \ref{theo: Phiint}. The following proposition is the main step in this proof.

\begin{prop}\label{lemma: diffeo}
 There exists $B = B(L,\Omega)>\lambda_1$ such that if $f$ satisfies  \hyph with $b<B$,  then the maps $PF_t:W \to Z$ are bi-Lipschitz homeomorphisms, uniformly in $t$.
\end{prop}
\begin{proof}
For  $w,\tilde w\in W$,  $t\in\RR$,   Proposition \ref{lemma: coercive} with  $u = w + t \phi_1, \tilde u = \tilde w + t \phi_1$ gives
\begin{equation}\label{coer}
 \|PF_t(w) - PF_t(\tilde{w})\|_Y \geq c \|w-\tilde{w}\|_X.
 \end{equation}
Hence $PF_t$ is injective and its image is closed.

We will first prove Proposition \ref{lemma: diffeo} under the additional hypothesis that
$f \in C^1$, which will let us use the implicit function theorem for $F$.

 Since the functions in $X$ are continuous in $\overline{\Omega}$, it is easy to see that $F: X \to Y$ is a $C^1$ function, with derivative at $u \in X$ given by
\[ DF(u): X \to Y \, \qquad  \langle DF(u),v\rangle  = - L v - f'(u) v, \quad v\in X. \]
Similarly, $PF_t:W \to Z$ is $C^1$ and, for every $w\in W$,
\[D(PF_t)(w):W \to Z \ , \quad \langle D(PF_t)(w), v \rangle = - Lv - P\big(f'(w+t\phi_1)v\big), \quad v\in W.\]

 We now show that  $ D(PF_t)(w) = -L - Pf'(w+ t \phi_1): W \to Z$ is an isomorphism for every $w\in W$. From \re{coer}, it is injective. To prove surjectivity, it suffices to show that it
is a Fredholm operator of index 0.
Recall that we  assume $a=0$, by \hyph, and thus $\lambda_1 = \lambda_1(L, \Omega)>0$, $L:X \to Y$ is an isomorphism (Theorem \ref{theo: exist}) and hence the restriction $L|_{W}:W \to Z$ is an isomorphism too. Also, the operator $v \in W  \mapsto P\big(f'(w+t\phi_1)v\big) \in Z$ is compact, since $P:Y\to Z$ is continuous, $f^\prime$ is bounded, and $W$ is compactly embedded in $Z$.

Thus $PF_t: W \to Z$ is a local diffeomorphism  so, from the inverse function theorem, its image is open. Since the image is also closed, $PF_t$ is surjective, hence, bijective. By \re{coer} the inverse of $PF_t$ is Lipschitz, uniformly in $t \in \RR$. Finally, by \hyph and the definition of the projection $P$ it is trivial to check that
$$
\|PF_t(w)-PF_t(\tilde w)\| \le C \|w-\tilde w\|,$$
for some constant $C$ which does not depend on $t$. \smallskip

For the general case of a Lipschitz  function  $f$ satisfying \hyph, we approximate $f$ by smooth functions $f_k: \RR \to \RR$ which also satisfy \hyph and converge uniformly to $f$ as $k\to\infty$. For instance,  the bump function $\psi_\delta : \RR\to \RR$,
\[\psi_\delta(x) = \frac{1}{\delta}\psi(x/\delta),
\]
with $\psi(x) = \chi_{[-1,1]}(x) \exp\left((|x|^2-1)^{-1}\right)$
yields the smooth functions
\begin{equation}\label{eq: fdelta}
f_\delta(x):=\int_\RR f(s)\psi_{\delta}(x-s) ds = \int_\RR f(x-s)\psi_{\delta}(s)ds.
\end{equation}
Since $f$ is uniformly continuous, $f_\delta \to f$ uniformly as $\delta=1/k\to0$.

The associated maps $F_k: X \to Y$ , $u\mapsto -Lu - f_k(u)$ are smooth and converge uniformly to $F$. Indeed, if $\epsilon > 0$ and $N$ is such that for all $k\geq N$ and $s\in\RR$ we have $|f_k(s) - f(s)|<\epsilon / |\Omega|^{\frac{1}{n}}$, then
\[\|F_k(u) - F(u)\|_Y = \|f_k(u) - f(u)\|_Y < (\epsilon / |\Omega|^{\frac{1}{n}})|\Omega|^{\frac{1}{n}}=\epsilon \ . \]

Thus the  maps $PF_{k,t}(w) = PF_k(w + t \phi_1)$  are smooth diffeomorphisms which converge uniformly to an injective map $PF_t$ with a closed image. Take $z \in Y$ and $w_k \in W$ for which $PF_{k,t}(w_k) =z$. Then we have
\[
\|z - PF_t(w_k)\|_Y=\|PF_{k,t}(w_k) - PF_t(w_k)\|_Y   \to 0\quad\mbox{ as }\;k\to\infty.\]
As the image of $PF_t$ is closed, $z$ is in the image of $PF_t$, i.e. $PF_t$ is surjective.
Uniform Lipschitz continuity for the inverses  $(PF_t)^{-1}:Z \to W$ again follows from \re{coer}.
\qed
\end{proof}

\bigskip

We are ready to prove Theorem \ref{theo: Phiint}.

\begin{proof}\noindent
From the previous proposition, the maps $\Psi$ and $\Phi=\Psi^{-1}$ are well defined. To see that  $\Psi$ is  Lipschitz, take $u = w + t \phi_1, \tilde u = \tilde w + \tilde t \phi_1 \in W \oplus V$ and compute:
\[ \|\Psi(u) - \Psi(\tilde u) \|_Y \le C \big(\| PF_t(w) - PF_{\tilde t} (\tilde w) \|_Y + | t - \tilde t | \big) \]
\[ \le C \big( \|F(u) - F(\tilde u)\|_Y + | t - \tilde t | \big)
 \le C \big( \|u - \tilde u \|_Y +  | t - \tilde t | \big)  \ . \]

To show that $\Phi$ is Lipschitz, for $t, \tilde{t} \in \RR$ and $z, \tilde{z}\in Z$,
\[ \|\Phi(z+t\phi_1) - \Phi(\tilde{z}+\tilde{t} \phi_1)\|_X  \leq \|(PF_t)^{-1}(z) - (PF_{\tilde{t}})^{-1}(\tilde{z} ) + (t-\tilde{t}) \phi_1\|_X \]
\[ \leq C \|z - \tilde{z} + (t-\tilde{t})\phi_1\|_Y\leq C (\|z - \tilde{z} + |t-\tilde{t}|)  \ , \]
where  for the second inequality we use Proposition \ref{lemma: coercive} with
\[ u = (PF_t)^{-1}(z) + t \phi_1, \quad \tilde u = (PF_{\tilde t})^{-1}(\tilde z) + \tilde t  \phi_1 \ . \]

From the definitions of $F$ and $\Phi$, $(F\circ\Phi)(z + t\phi_1)= z + \tilde h(z,t)\phi_1$, for some real number $\tilde h(z,t)$. Recall that $z \in Z$, so that $z$ is orthogonal to $\phi_1^*$. We must then have
\[
\tilde h(z,t)=\frac{\langle(F\circ \Phi)(z+t\phi_1),\phi_1^*\rangle}{\langle \phi_1, \phi_1^* \rangle} \ .
 \] \qed
\end{proof}

\subsection{Fibers and heights, properness of $F$}\label{sec: fiberbehaviour}

In this section we assume $f$ satisfies \hyph with $b<B$,  where $B >\lambda_1$ is defined by Proposition \ref{lemma: coercive}.

Fix $z \in Z$.
From the definition of $\Psi$ and the results from the previous section, every horizontal affine subspace $W+t \phi_1 $ is taken by $F$ to a surface $F(W+t \phi_1)$ which projects homeomorphically onto $Z$. In particular, this surface meets each line $\{ z + h \phi_1, h\in\mathbb{R} \} \subset Y$ at a single point $z + \tilde h(z,t) \phi_1$, which is the image of  a point $w(z,t) + t \phi_1 \in W \oplus V$.
Thus for each $z\in Z$ we can define the {\it fiber}
\[u_z(t)=u(z,t):=w(z,t) + t\phi_1 = \Phi(z+t\phi_1)= \Psi^{-1}(z+t\phi_1)\]
as the inverse of the slanted line  $\{ z + t \phi_1, t\in\mathbb{R} \} \subset Y$. In this way we also  define the {\it height  function} $\tilde h=\tilde h(z,t)$, by
\[ F(u(z,t)) =-Lu(z,t) -f(u(z,t)) = z + \tilde h(z,t)\phi_1 . \]

We rephrase some  Lipschitz properties of $F$ and $\Phi$ from the previous section.

\begin{prop}\label{prop: preimageoffiber}
For every $z\in Z$, the map $t\mapsto u(z,t)=\Phi(z+t\phi_1)$ is Lipschitz uniformly in $z$.
The height $\tilde h(z,t)$ is Lipschitz in both $z$ and $t$.
The equation $F(u)=g = z_g + t_g \phi_1\in Z \oplus V$ has as many solutions as the equation $\tilde h(z_g,t)=t_g$, for the unknown $t\in \RR$.
\end{prop}


\begin{prop}\label{prop: h(z,t)infinitylines}

As $|t| \to \infty$, $\tilde h(z,t) \to -\infty$ uniformly in $z\in Z$.
\end{prop}
\begin{proof}
We expand the expression for $\tilde h(z,t)$ in Theorem \ref{theo: Phiint}, using $u(z,t) = w(z,t) + t \phi_1 $, $w\in W$, $Lw\in Z$:
\begin{align}
\tilde h(z,t) & = \frac{\langle F(u(z,t)), \phi_1^*\rangle}{\langle \phi_1, \phi_1^*\rangle} \nonumber \\
& = \frac{\langle -Lw(z,t) - tL\phi_1, \phi_1^* \rangle}{\langle \phi_1, \phi_1^*\rangle} - \frac{\langle f(u(z,t)),\phi_1^*\rangle}{\langle \phi_1, \phi_1^*\rangle} = \lambda_1 t - \frac{\langle f(u(z,t)),\phi_1^*\rangle}{\langle \phi_1, \phi_1^*\rangle} \nonumber \\
& \le  \lambda_1 t +M \frac{\langle 1, \phi_1^*\rangle}{\langle \phi_1, \phi_1^*\rangle}-b\frac{\langle w(z,t),\phi_1^*\rangle +t\langle \phi_1, \phi_1^* \rangle}{\langle \phi_1, \phi_1^*\rangle}
 \le  (\lambda_1-b)t +C \label{eq: h},
\end{align}
where we used \hyph.
Since $\lambda_1 < {b}$, for $t\to +\infty$ we have $\tilde h(z,t) \to -\infty$. The bound does not depend on $z\in Z$, implying uniform convergence. The case $t\to -\infty$ is similar: replace $b$ by $a=0$ in \re{eq: h}, again by \hyph.
\qed
\end{proof}\medskip

\begin{prop}\label{prop:properness}
The map $F: X \to Y$ is proper.
\end{prop}

\begin{proof}
From Theorem \ref{theo: Phiint}, it suffices to establish the properness of
\[ \tilde F: Z \oplus V \to Z \oplus V \ , \quad  (z,t) \mapsto (z,\tilde  h(z,t)) \ .\]
Now, if $(z_k, \tilde h(z_k,t_k))$ is a convergent sequence then $(z_k,t_k)$ is precompact, since by Proposition \ref{prop: h(z,t)infinitylines} the sequence $\{ t_k\}$ is bounded.
\qed
\end{proof}\medskip

Next, we show that fibers at infinity are essentially parallel to $\phi_1$, that is, $w(z,t)$ is $o(t)$ as $|t|\to\infty$. Here we use the convexity of $f$.

\begin{lemma}\label{lemma: fiberinfinity}
For every $z\in Z$,
\[\lim_{|t|\to \infty} \|\frac{w(z,t)}{t}\|_X  =  \lim_{|t|\to \infty}\big\|\frac{u(z,t)}{t}-\phi_1\big\|_X=0.\]
\end{lemma}
\begin{proof}Fix $z\in Z$. By Proposition \ref{lemma: coercive}, for some $C>0$,
\begin{align*}
o(t) + \big\|\frac{PF_t(0)}{t}\big\|_Y &\ge\big\|\frac{z}{t} - \frac{PF_t(0)}{t}\big\|_Y\\
 &= \frac{1}{|t|}\|PF_t(w(z,t)) - PF_t(0)\|_Y \geq C\big\|\frac{w(z,t)}{t}\big\|_X,
 \end{align*}
so it suffices to prove that $$\frac{1}{|t|}\|PF_t(0)\|_Y =\frac{1}{|t|} \| Pf(t\phi_1(x)) \|_Y \to 0\qquad\mbox{ as }\;t\to \pm \infty.$$

Say $t\to +\infty$. Since $f$ is convex,
$(f(t) - f(0))/t$
is nondecreasing and bounded (by \hyph), hence convergent to some number $\tilde  b\le b$. In the limit,  the expression
\[ \frac{f(t\phi_1(x))}{t} =  \frac{f(t\phi_1(x))}{t\phi_1(x)} \ \phi_1(x) \]
converges pointwise to $ \tilde b  \phi_1(x)$, whose projection is the origin. The result follows by dominated convergence.
\qed
\end{proof}

\subsection{Proof of Proposition \ref{lemma: coercive}}\label{sec: theoremcoercive}

 The proposition is proved if we find numbers $\rho, c_0\in (0,1]$, depending only on $L$ and $\Omega$, such that if $f$ satisfies \hyph with $b = \lambda_1 + \rho$, then for every $u, \tilde u\in X$,
\begin{equation}\label{state}
\|\Psi(u)-\Psi(\tilde u )\|_Y \ge c_0 \ \|u-\tilde u \|_X.
\end{equation}
We use the product norms $\|u\|= \|w\|+ |t|\|\phi_1\|$, if $u= w+t\phi_1$, $w\in \mathrm{vect}(\phi_1^*)^\perp$ and normalize $\phi_1$ so that $\|\phi_1\|_X=1$.

Fix $u=w+t\phi_1, \tilde u= \tilde w  +\tilde t \phi_1\in X$,  $u\not =\tilde u$. By the definition of $\Psi$,
$$
\Psi(u)-\Psi(\tilde u ) = L(w-\tilde w) + P(f(u)-f(\tilde u)) + (t-\tilde t) \phi_1.
$$

Set
$$
{v}:=\frac{w-\tilde w}{\|w-\tilde w \|_X}\in W, \qquad \tau := \frac{|t-\tilde t|}{\|u-\tilde u \|_X}\in[0,1],
$$
and
\begin{align}
\psi &:= (1-\tau) L{v} + \frac{P(f(u)-f(\tilde u))}{\|u-\tilde u \|_X}\label{defpsi}\\
&= \frac{L(w-\tilde w)}{\|u-\tilde u \|_X} +\frac{P(f(u)-f(\tilde u))}{\|u-\tilde u \|_X}\nonumber
\end{align}
 With this notation, statement \re{state}, equivalent to Proposition \ref{lemma: coercive}, becomes
\begin{equation}\label{state1}
\|\psi\|_{L^p(\Omega)}  + \tau \ge c_0(L,\Omega).
\end{equation}

From now on we assume that $\tau\le 1/2$ (else  \re{state1} holds with $c_0=1/2$). Set
\[q(x):=
\left\{ \begin{array}{ccc}  \displaystyle\frac{f(u(x))-f(\tilde u(x))}{u(x)-\tilde u(x)} & \mbox{if} & u(x) \neq \tilde u(x) \medskip \\
0 & \mbox{if} & u(x) = \tilde u(x).
\end{array}
\right.
\]
From Lemma 1.1 in \cite{BNV}, $\lambda_1 \le C(L,\Omega)$. Thus \hyph implies $q\in L^\infty(\Omega)$, with
$$
0\le q\le \lambda_1 + \rho\le C_1(L,\Omega).
$$

Observe that
\begin{equation}\label{qux}
\frac{f(u)-f(\tilde u)}{\|u-\tilde u \|_X} = q(x) \left( (1-\tau) v + \tau\phi_1\right),
\end{equation}
and hence, for $C_2 = C_2(L,\Omega) = \|P\| \ C_1$,
\begin{equation}\label{pff}
\left\| \frac{P(f(u)-f(\tilde u))}{\|u-\tilde u \|_X} \right\|_Y \le C_2 \left( \|{v}\|_Y +\tau\right).
\end{equation}

We now apply the classical $W^{2,p}$-estimate (see for instance Theorem 9.13 in \cite{GT}) to \re{defpsi}, seen as an elliptic equation satisfied by ${v}$. Thus, if $C_3 = C_3(L,\Omega)$ is the constant from that estimate, by using \re{pff} we get
 \begin{align}
 1 &= \|{v}\|_X
 \le \frac{C_3}{1-\tau} \left( \|\psi\|_Y + C_2 (\|{v}\|_Y + \tau) + \|{v}\|_{L^\infty(\Omega)}\right) \nonumber \\
 &\le C_4 (\|\psi\|_Y +\tau + \|{v}\|_{L^\infty(\Omega)} ) \label{oci1}
 \end{align}
 for some $C_4=C_4(L,\Omega)$, where we also used  $\|{v}\|_Y = \|{v}\|_{L^p(\Omega)} \leq |\Omega|^{1/p} \|{v}\|_{L^\infty(\Omega)}$.

From now on we suppose $\|\psi\|_Y +\tau\le 1/(2C_4)$ (else \re{state1} holds, by setting $c_0 = 1/(2C_4)$). Then by \re{oci1}
 \begin{equation}\label{linf}
 \|{v}\|_{L^\infty(\Omega)} \ge \frac{1}{2C_4}=:c_1.
 \end{equation}

 On the other hand we also have, by the embedding $X \hookrightarrow C^{0,\alpha}(\Omega)$ for some fixed $\alpha<1$, that, for some $C_5=C_5(\Omega)$,
 \begin{equation}\label{calf}
 \|{v}\|_{C^\alpha(\Omega)} \le C_5 \|{v}\|_{X} = C_5.
 \end{equation}

If $p>n$ we have more, since then $X \hookrightarrow C^{1,\alpha}(\Omega)$ for $\alpha \in (0,1-n/p)$, and
 \begin{equation}\label{calf1}
 \|{v}\|_{C^{1,\alpha}(\Omega)} \le C_5^\prime \|{v}\|_{X} = C_5^\prime.
 \end{equation}
This estimate and ${v}=0$ on $\partial \Omega$ imply that  ${v}/d$ is H\"older continuous:
 \begin{equation}\label{calf2}
\left\|\frac{{v}}{d}\right\|_{C^{\alpha}(\Omega)} \le C_5^{\prime\prime},  \quad \hbox{for} \quad d(x)=\dist(x,\partial \Omega) \ .
 \end{equation}

Alternatively, the last estimate can be deduced in a standard fashion from the general Harnack inequality for ${v}/d$, proved in  \cite{Si1}. Clearly, \re{calf2} holds if ${v}$ is replaced by its positive or negative parts ${v}^+$ and ${v}^-$, which are compositions of~${v}$ and a Lipschitz function with Lipschitz constant equal to one.\smallskip

 We  now establish a useful property of ${v}$.

\begin{lemma}\label{lembound} There exist constants $\epsilon, \nu>0$ depending only on $L$ and $\Omega$, and subdomains $\omega_1, \omega_2\subset \Omega$ with measures $|\omega_1|, |\omega_2|\ge \nu$, such that $${v}\ge \epsilon\quad\mbox{in }\; \omega_1,\qquad\mbox{and}\qquad {v}\le -\epsilon \quad\mbox{in }\; \omega_2.$$
\end{lemma}

\noindent{\it Proof}. We first record the following \smallskip

Fact. If $g\in C^{\alpha}(\Omega)$ is such that $\|g\|_{C^{\alpha}(\Omega)}\le A$ and $g(x_0) \ge a>0$ (resp. $g(x_0) \le -a<0$) for some $x_0\in \overline{\Omega}$, then
$$g\ge \frac{a}{2}\;\left(\mbox{resp. } g\le -\frac{a}{2}\right)\quad \mbox{in }\;B_\nu(x_0)\cap\Omega,\qquad\mbox{where }\; \nu= \left(\frac{a}{2A}\right)^{1/\alpha}.
$$
This is immediate from the definition of the H\"older seminorm
$$
g(x_0)-g(x)\le \|v\|_{C^{\alpha}} |x-x_0|^\alpha, \qquad\mbox{i.e.}\qquad g(x)\ge g(x_0)-\|g\|_{C^{\alpha}} |x-x_0|^\alpha.
$$

We prove Lemma \ref{lembound}. By \re{linf}, there exists $x_1\in \Omega$ such that either ${v}(x_1)\ge c_1$ or ${v}(x_2)\le -c_1$. Say the first happens. Then the  fact above and \re{calf} imply ${v}\ge \epsilon_1 = c_1/2$ in $\omega_1=B_{\bar\nu_1}(x_1)\cap\Omega=B_{\bar\nu_1}(x_1)$, where $\bar\nu_1=(c_1/(2C_5))^{1/\alpha}$. It is clear that
$$
|\omega_1|=|B_{\bar\nu_1}(x_1)|\ge \nu_1>0,
$$
for some $\nu_1$ which depends only on $\bar\nu_1$ and $n$, i.e. on $L$ and $\Omega$. \smallskip

Recall that ${v}\in Z$, which means that $\langle {v},\phi_1^*\rangle=0$ where  $\phi_1^*>0$ is the principal eigenfunction of the dual operator. In other words,
$\int_\Omega {v}^+\phi_1^*=\int_{\Omega}{v}^-\phi_1^*$. We assume $\phi_1^*$ is normalized so that $\int_\Omega \phi_1^*=1$, and estimate
\begin{align*}
 \sup_{\Omega} {v}^- \geq \int_\Omega {v}^-\phi_1^* &= \int_\Omega v^+ \phi_1^*\\ & \geq \frac{c_1}{2} \int_{\omega_1} \phi_1^*
 \ge \frac{c_1}{2}\inf_{x\in\Omega} \int_{B_{\bar\nu_1}(x)\cap\Omega}\phi_1^*=:c_2.
\end{align*}
Note that the positive constant $c_2$ depends only on $c_1$, $\nu_1$, $\Omega$, and $\phi_1^*$, and therefore only on the operator $L$ and the domain $\Omega$. However, because of the rather obscure behaviour of $\phi_1^*$ we do not know how to prove in this generality that $c_2$ is bounded below by a constant which depends only on bounds on the coefficients of $L$, let alone exhibit such a lower bound.

On the other hand, by working a bit more, we will now show that an explicit lower bound for $\sup_{\Omega} {v}^-$ can be obtained if $p>n$, in terms of the constants in the basic elliptic estimates (the ABP inequality, the various forms of the Harnack inequality and the regularity estimates).

We introduce the auxiliary function $\zeta\in X$, the solution of
\begin{equation}\label{eqi4}
\left\{ \begin{array}{rcccc}
 L\zeta&=& -\chi(x) &\mbox{in}& \Omega\\
 \zeta&=& 0 &\mbox{on}& \partial\Omega,
 \end{array}
 \right.
 \end{equation}
 where $\chi(x)=\chi_{\omega_1}(x)$ denotes the indicator function of the set $\omega_1= B_{\nu_1}(x_1)$.

By  applying Theorem \ref{qsmp} to \re{eqi4} we get, for some $\bar c=\bar c(L,\Omega)$,
$$
\zeta\ge \bar c|\omega_1|^{1/\varepsilon}\,d\ge \bar c \nu_1^{1/\varepsilon} \,d\qquad\mbox{in }\; \Omega.
$$

Normalize now $\phi_1^*$ so that $\langle \phi_1^*,d\rangle=\int_\Omega \phi_1^*d=1$. We have the chain of estimates
\begin{align*}
 \sup_{\Omega} \frac{v^-}{d} &\geq \int_\Omega \frac{v^-}{d} \phi_1^*d= \int_\Omega v^- \phi_1^* = \int_\Omega v^+ \phi_1^*\\ & \geq \frac{c_1}{2} \int_{\omega_1} \phi_1^*
 = \frac{c_1}{2} \langle \phi_1^*,\chi\rangle
=  \frac{c_1}{2} \langle \phi_1^*,-L\zeta\rangle=\frac{c_1}{2} \langle -L^*\phi_1^*,\zeta\rangle  \\
 &=\lambda_1 \frac{c_1}{2}\int_\Omega \phi_1^*\zeta
 \ge \lambda_1 \frac{c_1}{2}\bar c \nu_1^{1/\varepsilon} \int_\Omega \phi_1^*d
 = \lambda_1 \frac{c_1}{2}\bar c \nu_1^{1/\varepsilon}:=c_2.
\end{align*}

Thus $\sup_\Omega \frac{{v}^-}{d} >c_2=c_2(L,\Omega)>0$. We now apply the fact above to $\frac{{v}^-}{d}$, by using \re{calf2}, to find a point $x_2\in \overline{\Omega}$ and some $\bar \nu_2>0$ such that $\frac{{v}^-}{d} >c_2/2$ in $\hat\omega_2=B_{\bar\nu_2}(x_2)\cap\Omega$. Set $$\omega_2=\hat\omega_2\cap\{x\in\Omega\::\:d(x)\ge \bar\nu_2/2\}.
$$
The measure of $\omega_2$ is controlled below by a constant $\nu_2>0$ which depends only on $\Omega$ and $\bar\nu_2$, and we have
$$
{v}^-\ge \frac{c_2}{2}d\ge \frac{c_2\bar\nu_2}{4}\quad\mbox{in }\;\omega_2.
$$
Lemma \ref{lembound} is proved. \qed\medskip

We continue with the proof of Proposition \ref{lemma: coercive}. Define
$$
\rho : =\min\left\{1,\frac{\eta_1}{2},\frac{\eta_2}{2}\right\},
$$
where $\eta_i=\eta_i(L,\Omega)>0$ is determined by Theorem \ref{theo: lambda1measure}, applied with $\Omega^\prime= \Omega\setminus \bar\omega_i$ ($\omega_i$ are given by Lemma \ref{lembound}).

By the definition of $P$ there exists $s\in \mathbb{R}$ such that (recall also \re{qux})
\begin{align*}
P\frac{f(u)-f(\tilde u)}{\|u-\tilde u \|_X} &= \frac{f(u)-f(\tilde u)}{\|u-\tilde u \|_X} +s\phi_1 \\
&= q(x) \left( (1-\tau) {v} + \tau\phi_1\right)+s\phi_1.
\end{align*}
Then \re{defpsi} can be written as
\begin{equation}\label{poi1}
L{v}+q(x) {v} = \frac{1}{1-\tau} \psi - \frac{\tau}{1-\tau}q(x) \phi_1 -s\phi_1.
\end{equation}

Assume first that $s\le0$. 

Since $q\le \lambda_1+\rho\le C_1(L,\Omega)$, $\|{\phi_1}\|_{L^\infty(\Omega)} \le C_5 \|{\phi_1}\|_{X} = C_5$, we get from \re{poi1}
\begin{equation}\label{poi}
L{v} + (\lambda_1+\rho) {v} \ge -2|\psi| -2 C_5(\lambda_1+\rho)\tau \ge -C_6(|\psi|+\tau)\quad\mbox{in }\;\Omega.
 \end{equation}
 Set
 $$
 \widetilde{L} = L+\lambda_1+\rho, \qquad \xi=C_6(|\psi|+\tau).
 $$
To summarize, we have
\begin{equation}\label{eqii4}
\left\{ \begin{array}{rcccc}
\widetilde{L}{v}&\ge& -\xi &\mbox{in}& \Omega\\
{v}&\ge& \epsilon  &\mbox{in}& \omega_1\\
{v}&\le& -\epsilon  &\mbox{in}& \omega_2\\
 {v}&=& 0 &\mbox{on}& \partial\Omega.
 \end{array}
 \right.
 \end{equation}

In addition,
\begin{align*}
\lambda_1(\widetilde{L}, \Omega\setminus \bar\omega_i)&= \lambda_1({L}, \Omega\setminus \bar\omega_i) - (\lambda_1+\rho)\\
&\ge \lambda_1+\eta_i - (\lambda_1+\rho)\ge \frac{\eta_i}{2},
\end{align*}
by the choice of $\rho$ and Theorem \ref{theo: lambda1measure}. Hence by Theorem \ref{theo: exist} we can solve the problem
\begin{equation}\label{eqi5}
\left\{ \begin{array}{rcccc}
\widetilde{L}\zeta&=& \xi &\mbox{in}& \Omega\setminus \bar\omega_2\\
 \zeta&=& 0 &\mbox{on}& \partial(\Omega\setminus \bar\omega_2),
 \end{array}
 \right.
 \end{equation}
and obtain
$$
\|\zeta\|_{L^\infty(\Omega\setminus \bar\omega_2)} \le \frac{C_{ABP}}{\eta_2/2}\|\xi\|_{L^n(\Omega)}\le \frac{C_{ABP}}{\eta_2/2} |\Omega|^{1/n-1/p}\|\xi\|_{L^p(\Omega)}=:C_7\|\xi\|_{L^p(\Omega)}.
$$

Assume by contradiction that
\begin{equation}\label{eqi7}
\|\xi\|_{L^p(\Omega)}<\frac{\epsilon}{C_7}.
\end{equation}
Then the function $v={v}+\zeta$ satisfies
$$
\widetilde{L}v\ge 0\quad\mbox{in }\; \Omega\setminus \bar\omega_2, \qquad v\le 0 \quad\mbox{on }\; \partial(\Omega\setminus \bar\omega_2), \qquad v>0 \quad\mbox{in }\; \omega_1\subset\Omega\setminus \bar\omega_2 .
$$
Thus Proposition \ref{eqdef} implies $\lambda_1(\widetilde{L}, \Omega\setminus \bar\omega_2)\le0$, contradicting $\lambda_1(\widetilde{L}, \Omega\setminus \bar\omega_2)\ge \eta_2/2$.

Hence \re{eqi7} fails, which is what we wanted to prove, since
$$
C_6(\|\psi\|_{L^p(\Omega)}+ |\Omega|^{1/p}\tau)\ge \|\xi\|_{L^p(\Omega)} \ge \frac{\epsilon}{C_7}
$$
implies \re{state1} by taking $$
c_0:=\min\left\{ \frac{1}{2}, \frac{1}{2C_4}, \frac{\epsilon}{C_6C_7\max\{1,|\Omega|^{1/p}\}}\right\}.
$$

If $s\ge0$, instead of \re{poi} we have
$$
L(-{v}) + (\lambda_1+\rho) (-{v}) \ge -C_6(|\psi|+\tau)
$$
so we can repeat the same argument, interchanging $\omega_1$ and $\omega_2$.

Proposition \ref{lemma: coercive} is proved. \qed

\section{Proof of Theorem \ref{theo: strictconvexInt}}\label{chap: convexLipschitz}

Take $B= B(L,\Omega)$ in the hypothesis of Theorem \ref{theo: strictconvexInt} to be  the minimum of the constants $\widetilde{B}$ and $B$, defined  in Proposition \ref{prop: principal} and Proposition~\ref{lemma: coercive}, respectively.

\begin{prop}\label{prop:threepoints} Under the hypotheses of Theorem \ref{theo: strictconvexInt},
no point of $Y$ has three preimages under $F$.
\end{prop}
\begin{proof} Such preimages would have to lie in the same fiber, that is, for some $z\in Y$ there exist $t_1 < t_2 < t_3$ and $u_i = u_z(t_i) = w_i + t_i \phi_1 \in W_X \oplus V$ with $F(u_i) = z + t \phi_1$ for a common height $t$. Then
\[ -L( u_2 - u_1) - (f(u_2) - f(u_1)) = 0 \ , \quad - L(u_3 - u_2) - (f(u_3) - f(u_2)) = 0. \]
We consider the  potentials
\[V_{i,j}(x):=
\left\{ \begin{array}{ccc}  \displaystyle\frac{f(u_i(x))-f(u_j(x))}{u_i(x)-u_j(x)} & \mbox{if} & u_j(x) \neq u_i(x) \smallskip \\
0 & \mbox{if} &  \ u_i(x)=u_j(x).
\end{array}
\right.
\]
Clearly $a= 0 \le V_{i,j} \le b < B $, and
\begin{equation}\label{ret} \big( -L - V_{2,1} \big) ( u_2 - u_1) = 0 \ , \quad \big( -L - V_{3,2} \big) ( u_3 - u_2) = 0 \qquad\mbox{in }\;\Omega.
\end{equation}

By Proposition \ref{prop: principal},  $u_3-u_2$ and $u_2 -u_1$ are principal eigenfunctions and do not change sign throughout $\Omega$. They are positive: indeed, as $\langle w_i - w_j,\phi_1^* \rangle =0$,
\[\langle u_i-u_j,\phi_1^*\rangle = \lambda_1 \langle (t_i-t_j)\phi_1,\phi_1^*\rangle = \lambda_1 (t_i-t_j)\langle \phi_1,\phi_1^*\rangle >0.\]
Hence $u_3 > u_2>u_1$ in  $\Omega$ and the potentials $V_{2,1}$, $V_{3,2}$ are continuous.

The convexity of $f$ implies that for any $\alpha_1,\alpha_2, \alpha_3\in \RR$
\[ \alpha_1 < \alpha_2 < \alpha_3 \Longrightarrow \frac{f(\alpha_2) - f(\alpha_1)}{\alpha_2 - \alpha_1} \ \le \ \frac{f(\alpha_3)-f(\alpha_2)}{\alpha_3 - \alpha_2} \ . \]
If equality happens, the function $f$ is affine  in $[\alpha_1, \alpha_3]$.

Set $\alpha_i = u_i(x)$ to obtain $V_{2,1}(x) \le V_{3,2}(x)$ for $ x \in \Omega$. From Proposition \ref{prop: monotonicity} and the fact that 0 is the principal eigenvalue of both $-L-V_{2,1}$ and $-L-V_{3,2}$, we must have $V_{2,1} \equiv V_{3,2}$ in $\Omega$. Thus, by the continuity of $u_i$ and $f$,
\[ f(t) = \alpha t + \beta \ , \quad \hbox{for} \ t \in I=[\inf u_1, \sup u_3] \ , \]
so that $V_{2,1} = V_{3,2} = \alpha $, and $\alpha=\lambda_1$ by \re{ret} and $u_2-u_1>0$.
Also $u_i(x) = 0$ for $x \in \partial \Omega$, so that $0 \in I$. This is a contradiction with $\mathbf{(C)}$.
\qed
\end{proof}
\smallskip

The second hypothesis in $\mathbf{(C)}$ is indeed necessary. If for instance $f(s) = \lambda_1 s + \beta$ in some interval   $(0,M)$ then for $t\in(0,M/\max \phi_1)$,
\[ F(t\phi_1) = -Lt\phi_1 - f(t\phi_1) = t\lambda_1\phi_1 - t\lambda_1\phi_1-\beta=-\beta,  \]
that is, the equation $F(u)=-\beta$ has a full segment of solutions.

\begin{prop}\label{fold} For a  Banach space $E$, consider the continuous proper map
\[ G: E \times \RR \to  E \times \RR \ , \quad  (e,t) \mapsto (e, g(e,t)) \ .\] 
Suppose that no point in $E \times \RR$ has three preimages under $G$. If some point has two preimages, $G$ is a global fold, that is, there are homeomorphisms \[ \sigma_1, \sigma_2: E \times \RR \to E \times \RR \, , \quad \sigma_1(e,t) = (e, g_1(e,t)) \, , \quad
\sigma_2(e,t) = (e, g_2(e,t)), \] such that $(\sigma_2 \circ G \circ \sigma_1)(e,t) = (e,-|t|)$. Otherwise $G$ is a homeomorphism.
\end{prop}

\begin{proof} The argument breaks in simple steps.

\medskip
\noindent {\bf Step 1:} Height functions $g(e,.)$ may have only four distinct topological types.

By properness, on each vertical line $ l_e=\{ (e,t),  t \in \RR\}$, $e \in E$,
\[ \lim_{t \to \infty} g(e,t) = \pm \infty \ , \quad \lim_{t \to -\infty} g(e,t) = \pm \infty \]
where the signs of both limits are not necessarily the same: there are two possibilities in which they are the same  and two in which they are different.
By hypothesis, there are no three points in a vertical line $l_e$ in the domain taken to the same point by $G $, and thus, after  changes of variables in the domain and counterdomain,  the height $s\to g(e,s)$ on each $l_e$ takes one of these four types,
\[ s \mapsto s , \quad s \mapsto -s , \quad s \mapsto |s| , \quad \hbox{or} \ s \mapsto - |s| \ . \]

\medskip
\noindent {\bf Step 2:} All heights of $G$ are of the same type.

By a connectivity argument, it  suffices to prove that,
for a fixed $e_0$, there is a neighborhood $N$ of $e_0$ for which all height functions ${g}(e,.)$ for $e \in N$ have the same limit. For example,  suppose by contradiction that $e_k \to e_0$ are such that
\[   (+) \quad \lim_{t \to \infty} {g}(e_k,t) = \infty \ , \quad \quad (-) \quad \lim_{t \to \infty} {g}(e_0,t) = - \infty \ . \]
By properness, the inverse of the compact set $K = \{ (e_k, 0)_k\} \cup \{ (e_0, 0) \}$ is a compact set, and therefore lies in   $\cup_k (e_k\times[-M,M])$ for some $M \in \RR$.  Then by the property (+), we must have $g(e_k,M+1) \ge 0$ for each $k$, and thus ${g}(e_0,M+1) \ge 0$, contradicting  (-).

\bigskip
If one ${g}(e,.)$ is of the first two types, ${G}$ is a homeomorphism. For the rest of the proof, we suppose that ${g}(e,.)$ is of the fourth type: in particular each function ${g}(e,.)$ is strictly unimodal, that is, $g(e,t)$ is strictly increasing for $t<t_0$, and strictly decreasing for $t>T$, for some $T=T(e)\in\RR$.

\medskip
\noindent {\bf Step 3:} Maxima of height functions, as well as points where they are attained, vary continuously across vertical lines.

Let $T(e)$ be the value of $t \in \RR$ at which $g(e,t)$ attains its maximum.
The map
$ e \in E \mapsto T(e) $
is well defined by the unimodality.

We show the continuity of $T$ at an arbitrary $e_0 \in E$. Set $T_0 = T(e_0)$ and take $\epsilon > 0$. For $L,R$ satisfying
\[ L < T_0 < R \, , \quad T_0 - L < \epsilon/2 \, , \quad R - T_0  < \epsilon/2 \, ,\]
take  $d > 0 $  so that
\[ g(e_0, T_0) - g(e_0, L) > d \, , \quad g(e_0, T_0) - g(e_0, R) > d \, . \]
By the continuity of $G$, there is $\delta >0$ for which, if $| e- e_0 | < \delta$, then
\[
| g(e, L ) - g(e_0, L)| \, , \quad | g(e, T_0 ) - g(e_0, T_0 )| \, , \quad | g(e, R ) - g(e_0, R)| < d / 3 \, . \]
If  $| e - e_0 | < \delta$, $g(e,T_0)$ is larger than $g(e,L)$ and $g(e,R)$:  for example, to estimate $g(e, T_0) - g(e, L)$, write
\[ g(e, T_0) - g(e_0, T_0) +  g(e_0, T_0) - g(e_0, L) + g(e_0, L) - g(e, L)  > -d/3 + d - d/3 = d/3 \,  .\]
Thus the point $T(e)$ where  $g(e,t)$ attains its maximum is still between $L$ and $R$, by the unimodality of $g(e,.)$. Since  $R- L < \epsilon$, we also have $|T(e) - T(e_0)| < \epsilon$. The continuity of the maximal value $z \in E \mapsto g(e, T(e))$ is now immediate.

\medskip
\noindent {\bf Step 4:} The global normal form.

The homeomorphisms
\[ \tau_1, \tau_2: E \times \RR \to E \times \RR \, , \quad \tau_1(e,t) = (e, t + T(e)) \, , \quad
\tau_2(e,s) = (e, s - g(e,T(e))) \]
yield the map $\tilde G = \tau_2 \circ G \circ \tau_1$, whose critical set $\tilde C$ together with its image $\tilde G(\tilde C)$ coincide with the horizontal plane $E \times \{0\}$. In addition, $\tilde G|_{\tilde C}$ is the identity. Moreover, the restrictions of $\tilde G$ on the half-spaces
\[\tilde G_-: E \times (-\infty, 0] \to E \times (-\infty, 0]  \quad \hbox{and} \quad \tilde G_+: E \times [0,\infty) \to E \times (-\infty,0]\]
are also homeomorphisms.

Set $\nu(z,t)=(z,-t)$. The juxtaposition of the maps $\tilde G_-$ and $\nu\circ \tilde G_+$ along $E \times \{0\}$ is a homeomorphism $j: E \times \RR \to E \times \RR$, and it is easy to see that
$\tilde G \circ j^{-1}: E \times \RR\to E \times \RR$ takes $(e,t)$ to $(e, -|t|)$.

The proposition is proved, setting $\sigma_1 = \tau_1\circ j^{-1}$, and $\sigma_2= \tau_2$.
\qed
\end{proof}

\bigskip
We finally complete the proof of Theorem \ref{theo: strictconvexInt}.\smallskip

\noindent{\bf Proof of Theorem \ref{theo: strictconvexInt}:}
Let $\tilde F = F\circ \Psi^{-1}:Z\times \RR\to Z\times \RR$ be the map defined in Theorem~\ref{theo: Phiint}. From Proposition \ref{prop:threepoints} no point has three preimages under $F$, and hence under $\tilde F$.

From the previous proposition, $\tilde F$ is either a homeomorphism or a global fold. It is not a homeomorphism, since from Proposition \ref{prop: h(z,t)infinitylines} on both extremes of each fiber there are points  which have the same image under $F$.

Let $\sigma_1$, $\sigma_2$ be the maps given by Proposition \ref{fold}, applied to $\tilde F$. Define the map $\tilde \psi : X=W\oplus \RR\phi_1 \to Y=Z\oplus\RR\phi_1$ by $\tilde \psi (w+t\phi_1) = -Lw +t\phi_1$.

Finally, we set
$$
\Phi_1 = \Psi^{-1}\circ \sigma_1 \circ \tilde \psi :X\to X, \qquad \Phi_2 = \sigma_2:Y\to Y.
$$
With this definition and Proposition \ref{fold}, we easily check that \re{maine} holds.

Obviously $\Phi_2$ leaves vertical lines invariant, by the definition of $\sigma_2$. To show the asymptotic property of $\Phi_1$, observe that by the definition of this map for each fixed $w\in W$ the point $\Phi_1(w+t\phi_1)$ is on the fiber generated by $z=-Lw$, and 
\begin{equation}\label{fini}
\Phi_1(w+t\phi_1) = \Psi^{-1}(-Lw,\hat t\,),
\end{equation}
where $\hat t= \hat t(w,t)$ is the number for which exists a point $\hat w \in W$ such that
\begin{equation}\label{fini1}
F(\hat w + (\hat t +c_1) \phi_1) = \left\{ \begin{array}{rclcl}
-Lw + (t-c_2)\phi_1&\mbox{ if }&t\le c_2,&\mbox{ and }& \hat t<0\\
-Lw - (t-c_2)\phi_1&\mbox{ if }&t\ge c_2,&\mbox{ and }& \hat t>0.
\end{array}
\right.
\end{equation}
Here $c_1$, $c_2$ are real constants whose values are irrelevant to our computation below (they depend only on $w$, and are related to the maximum of the height function on the fiber  generated by $z=-Lw$).  By \re{fini1} and the properness of $F$ it is clear that $\lim_{t\to\pm\infty} \hat t = \pm\infty$.

We are going to show that
\begin{equation}\label{fini2}
\lim_{t\to-\infty}\frac{t}{\hat t} = \lambda_1 - \tilde a, \qquad \lim_{t\to\infty}\frac{t}{\hat t} = \tilde b-\lambda_1,
\end{equation}
where $\tilde a := \lim_{s\to -\infty} \frac{f(s)}{s}<\lambda_1$, $\tilde b := \lim_{s\to \infty} \frac{f(s)}{s}>\lambda_1$ (see the proof of Lemma \ref{lemma: fiberinfinity}), from which we infer the asymptotics
\begin{equation}\label{fini3}
\lim_{t\to-\infty}
\frac{\Phi_1(w+t\phi_1)}{t}=\frac{1}{\lambda_1-\tilde a}\,\phi_1,
\qquad
\lim_{t\to\infty}
\frac{\Phi_1(w+t\phi_1)}{t}=\frac{1}{\tilde b-\lambda_1}\,\phi_1
\end{equation}
in $X$, thanks to \re{fini} and Lemma \ref{lemma: fiberinfinity}.

Exactly like in the proof of Lemma \ref{lemma: fiberinfinity} we can show that
$$
\lim_{|\hat t|\to \infty} \|\frac{\hat w}{\hat t}\|_X \le C \lim_{|\hat t|\to \infty} \frac{1}{|\hat t|} \| PF_{\hat t +c_1}(\hat w) -PF_{\hat t +c_1}(0)\| =0.
$$

Writing \re{fini1} in the form
$$
-L(\hat w + (\hat t + c_1)\phi_1) - V(x) (\hat w + (\hat t + c_1)\phi_1) =- Lw\pm (t-c_2)\phi_1
$$
where $V(x) = f(\hat w + (\hat t + c_1)\phi_1)/(\hat w+(\hat t + c_1)\phi_1)$ converges to $\tilde b$ as $t\to\infty$ (resp. to $\tilde a$ as $t\to-\infty$), multiplying by $\phi_1^*$ and integrating, dividing by $\hat t$ and letting $\hat t \to\pm\infty$, we arrive to \re{fini2}.  \qed

{

\parindent=0pt
\parskip=0pt
\obeylines

\bigskip

Boyan Sirakov, Carlos Tomei and André Zaccur, \bigskip

Departamento de Matem\'atica, PUC-Rio,
Rua Marques de Sao Vicente 225, Rio de Janeiro, RJ 22451-900, Brazil

}

\end{document}